\documentclass[12pt,reqno]{amsart}

\usepackage{amsmath,amsfonts,euscript,amscd,amsthm,amssymb,upref,graphics}
\usepackage[all]{xy}

\theoremstyle{definition}

\swapnumbers

\theoremstyle{plain}

\newtheorem{theorem}{Theorem}[section]
\newtheorem{proposition}[theorem]{Proposition}
\newtheorem{lemma}[theorem]{Lemma}

\newtheorem{paragItalic}[theorem]{}

\theoremstyle{definition}

\newtheorem{definition}[theorem]{Definition}

\newtheorem{parag}[theorem]{}

\newtheorem{notation}[theorem]{Notation}

\theoremstyle{remark}

{\begin{enumerate}\setlength{\itemsep}{#1}}{\end{enumerate}}
\newenvironment{enumerata}%
{\begin{enumerate}

}{\end{enumerate}}
\newenvironment{Enumerata}[1]%
{\begin{enumerata}\setlength{\itemsep}{#1}}{\end{enumerata}}

\newcommand{\supp}{	\operatorname{{\rm supp}}}

\newcommand{\Div}{	\operatorname{{\rm Div}}}

\newcommand{\Bs}{	\operatorname{{\rm Bs}}}

\newcommand{\setspec}[2]{\big\{\,#1\, \mid \,#2\, \big\}}

\newcommand{\Integ}{\ensuremath{\mathbb{Z}}}
\newcommand{\Nat}{\ensuremath{\mathbb{N}}}

\newcommand{\proj}{\ensuremath{\mathbb{P}}}
\newcommand{\bk}{{\ensuremath{\rm \bf k}}}

\newcommand{\bbL}{\ensuremath{\mathbb{L}}}
\newcommand{\bbG}{\ensuremath{\mathbb{G}}}

\newcommand{\Keul}{\EuScript{K}}

\newcommand{\isom}{\cong}
\renewcommand{\epsilon}{\varepsilon}
\renewcommand{\phi}{\varphi}
\renewcommand{\emptyset}{\varnothing}

\newcommand{\barr}{\overline}

\newcommand{\rien}[1]{}

\begin{document}
\renewcommand{\baselinestretch}{1.07}


\title[Linear systems of rational curves on rational surfaces]
{Linear systems of rational curves\\
 on rational surfaces}

\author{Daniel Daigle, Alejandro Melle-Hern\'andez}

\date{\today}

\address{Department of Mathematics and Statistics\\ University of Ottawa\\ Ottawa, Canada\ \ K1N 6N5}

\address{ICMAT (CSIC-UAM-UC3M-UCM) Dept.\ of Algebra, 
Facultad de Ciencias Matem\'aticas, Universidad Complutense, 28040, Madrid, Spain}

\email{ddaigle@uottawa.ca}\email{amelle@mat.ucm.es} 

\thanks{Research of the first author supported by a grant from NSERC Canada.}
\thanks{Research of the second author supported by the grants MTM2010-21740-C02-01 and
Grupo Singular CCG07-UCM/ESP-2695-921020.}

{\renewcommand{\thefootnote}{}
}

\dedicatory{Dedicated  to the memory of Vladimir Igorevich Arnold}


\maketitle
  
\vfuzz=2pt

Let $S$ be a projective nonsingular rational surface, over an algebraically closed field of characteristic zero.
We say that a linear system $\bbL$ on $S$ is \textit{rational\/}
if $\dim\bbL\ge1$ and the general member of $\bbL$ is
an irreducible rational curve.

Let $C \subset S$ be an irreducible curve.

Let $\Omega_{C}$ be the set of rational linear systems $\bbL$ on $S$ satisfying $C \in \bbL$.
Consider the minimal\footnote{This is the ``short'' resolution, not the ``embedded'' resolution.
See \ref{diouf90qwejdk} for details.} resolution of singularities $\pi : \widetilde S \to S$ of $C$,
let $\widetilde C$ be the strict transform of $C$ on $\widetilde S$,
and let ${\tilde \nu}(C)$ denote the self-intersection number of $\widetilde C$ in $\widetilde S$.
Then Theorem~\ref{yyvyvyvyvydyddsdfy} implies:
\begin{enumerate}

\item {\it $\Omega_C \neq \emptyset$ if and only if $C$ is rational and ${\tilde \nu}(C) \ge 0$.}

\end{enumerate}

Let $\bbL_C$ be the linear system on $S$ which is the image of $| \widetilde C |$
 by $\pi_*: \Div(\widetilde S)\to \Div(S)$
(so $C \in \bbL_C$).
Assuming that $\Omega_C \neq \emptyset$,
we show (Thm~\ref{yyvyvyvyvydyddsdfy}):

{\it
\begin{enumerate}
\addtocounter{enumi}{1}

\item For any linear system $\bbL$ on $S$,\ \  
$\bbL \in \Omega_C$ $\iff$ $\dim\bbL \ge1$ and $C \in \bbL \subseteq \bbL_C$.

\end{enumerate}}
This gives a complete description of $\Omega_C$, and we note in particular 
that $\Omega_C$ has a greatest element (namely, $\bbL_C$).
Continuing to assume that $\Omega_C \neq \emptyset$,
Theorem~\ref{yyvyvyvyvydyddsdfy} also shows that
$\dim\bbL_C = {\tilde \nu}(C)+1$ and that 
the minimal resolution of singularities of $C$
coincides with the minimal resolution of the base points of $\bbL_C$.

\medskip
The present paper may be viewed as a preamble to the forthcoming \cite{Dai_Mell_2},
in which we study linear systems associated to unicuspidal rational curves $C \subset \proj^2$. 
We remind the reader that all currently known curves of that type satisfy ${\tilde \nu}(C) \ge 0$, hence $\Omega_C \neq \emptyset$.
It is shown in \cite{Dai_Mell_2} that
if $C \subset \proj^2$ is a unicuspidal rational curve with singular point $P$ then:
(1) there exists a unique pencil $\Lambda_C$ on $\proj^2$
satisfying $C \in \Lambda_C$ and $\Bs( \Lambda_C ) = \{P\}$;
(2) $\Lambda_C$ is a rational pencil if and only if ${\tilde \nu}(C)\ge0$;
(3) if ${\tilde \nu}(C)\ge0$, then $\Lambda_C$ has a dicritical of degree $1$.
From these facts, the paper then derives consequences for the classification of 
unicuspidal rational curves $C \subset \proj^2$ satisfying ${\tilde \nu}(C) \ge 0$.

\bigskip
We would like to express our thanks to Professor Dolgachev, and also to the referee,
for their useful comments.

\bigskip
\noindent{\bf Conventions.} 
All algebraic varieties are over an algebraically closed field $\bk$ of characteristic zero.
Varieties are irreducible and reduced, so in particular all curves are irreducible and reduced. 
A divisor $D$ of a surface is {\it reduced\/} if $D = \sum_{i=1}^n C_i$ where $C_1, \dots, C_n$
are distinct curves ($n\ge0$). 


\section{Clusters on a surface}
\label{sdifrqiowekf}

We fix a projective nonsingular surface $S$ throughout this section.
We consider the set $S^*$ of points which are either 
points of $S$ or points infinitely near points of $S$.
The set $S^*$ comes equipped
with a partial order $\le$, called the \textit{natural order},
such that for $P,Q \in S^*$ we have $P<Q$ if and only if
$Q$ is infinitely near $P$.
The minimal elements of $(S^*,\le)$ are called the \textit{proper points\/}
of $S$, and are indeed in bijective correspondence with
the closed points of $S$. 
Note that the poset $(S^*,\le)$ is a classical object
(for instance it is called a ``bubble space'' in \cite{manin74} but has the order relation reversed).

A \textit{cluster\/} on $S$ is a (possibly empty)
finite subset $K \subset S^*$ such that,
given any $P,Q \in S^*$, if $P\le Q$ and $Q \in K$ then $P \in K$.
If $K$ is a cluster on $S$ then a \textit{subcluster\/} of $K$ is
any subset of $K$ which is itself a cluster on $S$.
Note that if $K$ is a cluster on $S$ then each minimal element of $K$
is a proper point of $S$.

The aim of this section is to fix the notations and terminologies for clusters
and to recall certain facts in that theory---there are no new results here.
Our main reference is the first chapter of \cite{Alberich},
and our notations and definitions are in general compatible with that text.

\begin{parag}  \label{;;;;;;;;;;fwo9eirij9}
Let $K$ be a cluster on $S$.
\begin{enumerata}

\item The \textit{blowing-up of $S$ along $K$} is denoted
$\pi_K : S_K \to S$.
Observe that if $K'$ is a subcluster of $K$ then
$K \setminus K'$ is a cluster on $S_{K'}$ and $\pi_K$ factors as
\begin{equation}  \label{drq293123q9uksdfjaw9e}
S_K \xrightarrow{ \pi_{K \setminus K'} }  S_{K'} \xrightarrow{ \pi_{K'} } S.
\end{equation}

\item Given a divisor $D \in \Div(S)$,
let $\widetilde D^K \in \Div(S_K)$ and $\barr D^K \in \Div(S_K)$ denote,
respectively, the strict transform and total transform of $D$ on $S_K$.

\item If $\bbL$ is a linear system on $S$ without fixed components and
 such that $\dim \bbL \ge1$,
let $\widetilde \bbL^K$ denote the strict transform of $\bbL$ on $S_K$.

\item Given $P \in K$,
one can define the corresponding exceptional curve $E_P$ as follows.
Consider the subcluster $K' = \setspec{ x \in K }{ x \le P }$ of $K$
and factor $\pi_K$ as in \eqref{drq293123q9uksdfjaw9e}.
Then $E_P \subset S_{K'}$ is the unique irreducible component of
the exceptional locus of $\pi_{K'}$ with self-intersection $(-1)$.
The strict transform (resp.\ total transform) of $E_P$ on $S_K$ is denoted
$\widetilde E_P^K \subset S_K$  (resp.\ $\overline E_P^K \in \Div( S_K )$).

\end{enumerata}
\end{parag}

\begin{parag}  \label{djfieieioqopalld}
\begin{enumerata}

\item \label{d99dr9g7g6g6hjg75hj7c478}
Given $P \in S^*$,
consider the blowing-up $\pi_{K^P} : S_{K^P} \to S$ of $S$
along the cluster $K^P = \setspec{ x \in S^* }{x < P}$, and note that
$P$ is a proper point of $S_{K^P}$.

\item Given $P \in S^*$ and a curve $C \subset S$,
let $e_P(C) \in \Nat$ denote the multiplicity of $C$ at $P$
(by definition, this is the multiplicity of
the curve $\widetilde C^{K^P} \subset S_{K^P}$ at the proper point
$P$ of $S_{K^P}$).
Extending linearly,
let $e_P(D) \in \Integ$ denote the multiplicity of a divisor $D \in \Div(S)$
at $P$. 

\item Given $P \in S^*$ and a linear system $\bbL$ on $S$
without fixed components and such that $\dim \bbL \ge1$, let
$e_P( \bbL ) \in \Nat$ denote the multiplicity of $\bbL$ at $P$
(by definition,
$ e_P( \bbL ) = \min\setspec{ e_P(D) }{ D \in \widetilde\bbL^{K^P} } $).
Note that the general member $D$ of $\widetilde\bbL^{K^P}$
satisfies $e_P(D) = e_P( \bbL )$.

\end{enumerata}
\end{parag}

\begin{parag}
A \textit{weighted cluster\/} on $S$ is a pair $(K,m)$ where $K$
is a cluster on $S$ and $m : K \to \Integ$ is any set map.
If $K'$ is a subcluster of $K$ and $m' : K' \to \Integ$ is the restriction
of $m$, we call $(K',m')$ a \textit{weighted subcluster\/} of $(K,m)$.
\end{parag}

\begin{parag}  \label {diouf90qwejdk}
Consider an effective divisor $D \in \Div(S)$.
\begin{enumerata}

\item Define the set
$K^D = \setspec{ P \in S^* }{ e_P(D)>1 }$ and note
that this is a finite set if and only if $D$ is reduced.

\item Assume that $D$ is reduced.
Then $K^D$ is a cluster on $S$,
called the \textit{cluster of singular points of $D$}.
If $K^D = \emptyset$, we say that $D$ is nonsingular.
The blowing-up $\pi_{K^D}: S_{K^D} \to S$ of $S$ along $K^D$
is called the \textit{minimal resolution of singularities\/} of $D$.
For an arbitrary cluster $K$ on $S$,
\begin{equation}  \label{dofj[q293i[lka;dfk;;}
\text{$\widetilde D^K$ is nonsingular $\iff$ $K^D \subseteq K$.}
\end{equation}

\item Continue to assume that $D$ is reduced.
If $e(D) : K^D \to \Integ$ denotes the map $P \mapsto e_P(D)$ then
we call $\Keul^D = (K^D, e(D))$ the 
\textit{weighted cluster of singular points of $D$}.

\end{enumerata}
\end{parag}

\begin{parag}  \label{;dslfq23jkwadll;}
Consider a linear system $\bbL$ on $S$ such that $\dim \bbL \ge1$ and
without fixed components.
\begin{enumerata}

\item The set $K_\bbL = \setspec{ P \in S^* }{ e_P(\bbL)>0 }$
is a cluster on $S$, called the \textit{cluster of base points of $\bbL$}.
The blowing-up $\pi_{K_\bbL}: S_{K_\bbL} \to S$ of $S$ along $K_\bbL$
is called the \textit{minimal resolution of the base points\/} of $\bbL$.
For an arbitrary cluster $K$ on $S$,
\begin{equation}  \label {sdkfjpq92391`]203'p;;;1}
\text{$\widetilde \bbL^K$ is base-point-free $\iff$ $K_\bbL \subseteq K$.}
\end{equation}
Let us also observe the following property of $K_{\bbL}$:
\begin{equation}  \label{sdfjpqw88888}
\begin{minipage}[t]{.8\textwidth}
For each $P \in K_\bbL$, if
$(\widetilde E_P^{K_{\bbL}})^2 = -1$ in $S_{K_\bbL}$ then 
$\widetilde E_P^{K_{\bbL}}$ is a horizontal curve (i.e., is not
included in the support of an element of $\widetilde\bbL^{K_{\bbL}}$).
\end{minipage}
\end{equation}

\item If $e(\bbL) : K_\bbL \to \Integ$ denotes the map
$P \mapsto e_P(\bbL)$ then
we call $\Keul_\bbL = (K_\bbL, e(\bbL))$ the
\textit{weighted cluster of base points of $\bbL$}.

\item We write $\Bs(\bbL) = \setspec{Q \in S}{ e_Q( \bbL ) > 0 }$ for
the \textit{base locus\/} of $\bbL$. Note that this is a set of proper
points of $S$, and is the set of minimal elements of $K_\bbL$.

\end{enumerata}
\end{parag}

\begin{parag}  \label {dkfjpqw9eklasdj}
Let $\Keul = (K,m)$ be a weighted cluster on $S$ and $D$ a 
divisor on $S$.
Let us use the notation $m = ( m_P )_{ P \in K }$ for the map $m$
and let $\pi_K : S_K \to S$ be the blowing-up of $S$ along $K$.
\begin{enumerata}

\item  The \textit{virtual transform of $D$ with respect to $\Keul$} is 
the divisor $\check D^{\Keul} \in \Div( S_K )$ defined by:
$$
\check D^{\Keul} = \barr D^{K} - \sum_{ P \in K } m_P \barr E_P^{K} .
$$

\item We say that \textit{$D$ goes through $\Keul$} if 
$\check D^{\Keul}$ is an effective divisor.
Note that if $D$ goes through $\Keul$ then $D$ is effective.

\item We say that \textit{$D$ goes through $\Keul$ effectively} if 
the following equivalent conditions are satisfied:
\begin{itemize}

\item $D$ is effective and  $e_P(D) = m_P$ for all $P \in K$

\item $D$ goes through $\Keul$ and $e_P(D) = m_P$ for all $P \in K$

\item $D$ is effective and $\check D^{\Keul} = \widetilde D^K$.

\end{itemize}
\end{enumerata}
\end{parag}

We leave it to the reader to verify assertions \ref{aaaaa8d7g63}--\ref{dfiupqwe;lkdj}, below.
To prove \ref{dfiupqwe;lkdj}(b), one uses characteristic zero Bertini Theorem.

\begin{parag} \label {aaaaa8d7g63}
\it 
Let $\Keul$ be a weighted cluster on $S$ and $D \in \Div(S)$.
If $D$ goes through $\Keul$, then $D$ goes through every weighted subcluster of $\Keul$.
\end{parag}

\begin{parag} \label {7vv6brrnfh}
\it 
Let $\Keul = (K,m)$ be a weighted cluster on $S$ and $D \in \Div(S)$.
Suppose that $D$ goes through $\Keul$, and that $e_P(D) \le m_P$ for all $P \in K$.
Then $e_P(D) = m_P$ for all $P \in K$.
\end{parag}

\begin{paragItalic}  \label {;sdkfj9238923-iklj}
Let $\bbL$ be a linear system on $S$ without fixed component and
such that $\dim\bbL \ge 1$.
For any $D \in \bbL$ and any cluster $K$ on $S$, the following are equivalent:
\begin{enumerata}

\item $\widetilde D^K \in \widetilde \bbL^K$

\item $e_P(D) = e_P( \bbL )$ for all $P \in K$

\item $e_P(D) \le e_P( \bbL )$ for all $P \in K$

\item $D$ goes through the weighted cluster $(K, e(\bbL))$ effectively,
where $e(\bbL) : K \to \Integ$ denotes the set map $P \mapsto e_P(\bbL)$.

\end{enumerata}
\end{paragItalic}

\begin{notation}  \label{fiouqp9384-9012}
If $\Keul = (K,m)$ is a weighted cluster, 
let $\Keul^{(>1)} = (K',m')$ be the pair 
defined by setting $K' = \setspec{ P \in K }{ m(P)>1 }$ and by
letting $m' : K' \to \Integ$ be the restriction of $m : K \to \Integ$ to $K'$.
\end{notation}

\begin{paragItalic}  \label {dfiupqwe;lkdj}
Let $\bbL$ be a linear system on $S$
without fixed component and such that $\dim\bbL \ge 1$.
\begin{enumerata}

\item  For any $D \in \bbL$, the following are equivalent:
\begin{enumerata}

\item $\Keul^D = \Keul_\bbL^{(>1)}$

\item $\widetilde D^{K_\bbL} \in \widetilde \bbL^{K_\bbL}$
and $\widetilde D^{K_\bbL}$ is nonsingular.

\end{enumerata}

\item \label{48993wg7gj5l3k}
The general member $D$ of $\bbL$ satisfies \mbox{\rm (a-i)} and \mbox{\rm (a-ii)},
and goes through $\Keul_\bbL$ effectively.

\end{enumerata}
\end{paragItalic}


\section{Rational linear systems on rational surfaces}

In this section, $S$ is a rational nonsingular projective surface.

\begin{definition} \label{88800112h}
A linear system $\bbL$ on $S$ is \textit{rational\/}
if $\dim\bbL\ge1$ and the general member of $\bbL$ is a rational curve.
\end{definition}

Given a curve $C \subset S$, it is interesting to ask whether there exists
a rational linear system $\bbL$ on $S$ satisfying $C \in \bbL$.
In this section we show that
the existence of $\bbL$ is equivalent to $C$ being rational and
of nonnegative type (cf.~\ref{minimalresolsing}).
When $C$ satisfies these conditions, we describe all rational linear systems containing $C$.

We begin by recalling some known facts (\ref{jhdhsgghgaqaqoaiaaqo} and \ref{dskjfwejf;akj}).

\begin{definition}  \label{kd545342100ko}
A pencil $\Lambda$ on $S$ is called a \textit{$\proj^1$-ruling\/}
if it is base-point-free and if its general member is isomorphic to a
projective line.
\end{definition}

The following fact is a consequence of a well-known result of Gizatullin
(see for instance \cite[Chap.~2, 2.2]{Miy_RatUnirat} or  \cite[Sec.~2]{Giz70}).
Note that Gizatullin's result is stronger than \ref{jhdhsgghgaqaqoaiaaqo},
as we are only stating the part of the result which we need.

\begin{lemma}[Gizatullin] \label {jhdhsgghgaqaqoaiaaqo}
Let $\Lambda$ be a $\proj^1$-ruling on $S$ and let $D \in \Lambda$.
\begin{enumerata}

\item Each irreducible component of $D$ is a nonsingular rational curve.

\item If $\supp(D)$ is irreducible then $D$ is reduced.

\item If $\supp(D)$ is reducible then there exists a $(-1)$-component $\Gamma$
of $\supp(D)$ which meets at most two other components of $\supp(D)$;
moreover, if $\Gamma$ has multiplicity $1$ in the divisor $D$ then
there exists another $(-1)$-component
of $\supp(D)$ which meets at most two other components of $\supp(D)$.

\end{enumerata}
\end{lemma}

\begin{lemma} \label {dskjfwejf;akj}
Consider $C \subset S$ such that $C \isom \proj^1$ and $C^2 \ge 0$.
\begin{enumerata}

\item $\dim |C| = C^2+1$ and $|C|$ is base-point-free.

\item For any linear system $\bbL$ on $S$ such that $C \in \bbL$ and $\dim\bbL\ge1$,
the general member of $\bbL$ is a nonsingular rational curve.

\item If $C^2=0$ then $|C|$ is a $\proj^1$-ruling.

\end{enumerata}
\end{lemma}

\begin{proof}
Assertions (a) and (c) are well known.
Let $\bbL$ be a linear system on $S$ such that $C \in \bbL$ and $\dim\bbL\ge1$, and
consider a general member $D$ of $\bbL$. Then $D$ is irreducible and reduced (because $\bbL$ has an element which
is irreducible and reduced)
and $p_a(D)=p_a(C)=0$ (because $D$ is linearly equivalent to $C$); so $D$ is a nonsingular rational curve.
\end{proof}


\medskip
Let us now turn our attention to the subject matter of this section, i.e.,
the problem of describing all rational linear systems containing a given curve.

\begin{definition} \label {minimalresolsing}
Let $C \subset S$ be a curve.
Consider the minimal resolution of singularities
$\pi = \pi_{K^C} : S_{K^C} \to S$ of $C$ (cf.\ \ref{diouf90qwejdk}),
and the strict transform $\widetilde C = \widetilde C^{K^C} \subset S_{K^C}$ of $C$.
Let ${\tilde \nu}(C)$ denote the 
self-intersection number of $\widetilde C$ in $S_{K^C}$.
If ${\tilde \nu}(C) \ge 0$, we say that $C$ is \textit{of nonnegative type}. 
We also define the set
$$
\bbL_C = \setspec{ \pi_*(D) }{ D \in | \widetilde C | }
$$
where $\pi_* : \Div( S_{K^C} ) \to \Div(S)$ is the homomorphism induced by $\pi = \pi_{K^C}$.
It is clear that $\bbL_C$ is a linear system on $S$, that $\dim\bbL_C = \dim | \widetilde C |$, and that $C \in \bbL_C$.
\end{definition}

\begin{lemma}  \label {bbbb3982edu}
Let $C \subset S$ be a rational curve.
\begin{enumerata}

\item $\dim\bbL_C \ge 1$ $\Leftrightarrow$ $\tilde\nu(C)\ge0$

\item If $\tilde\nu(C)\ge0$, then $\dim\bbL_C = \tilde\nu(C) + 1$.

\item If $\tilde\nu(C)\ge0$, then every
linear system $\bbL$ on $S$ satisfying $C \in \bbL \subseteq \bbL_C$ and $\dim\bbL\ge1$ is a
rational linear system.

\end{enumerata}
\end{lemma}

\begin{proof}
Let the notation
($\pi = \pi_{K^C} : S_{K^C} \to S$, 
$\pi_* : \Div( S_{K^C} ) \to \Div(S)$ and $\widetilde C = \widetilde C^{K^C} \subset S_{K^C}$)
be as in \ref{minimalresolsing}.
We have
$\dim\bbL_C = \dim | \widetilde C |$ and $\tilde\nu(C) = {\widetilde C}^2$,
so assertion (b) follows by applying \ref{dskjfwejf;akj}(a) to the nonsingular curve $\widetilde C$.
Part ``$\Leftarrow$'' of (a) follows immediately, and the converse is the observation that
$\dim| \widetilde C | \ge 1$ implies ${\widetilde C}^2\ge0$.

To prove (c), suppose that $\tilde\nu(C)\ge0$ and
consider a linear system $\bbL$ on $S$ satisfying $C \in \bbL \subseteq \bbL_C$ and $\dim\bbL\ge1$.
Then there exists a linear system $\bbL'$ on $S_{K^C}$ satisfying $\widetilde C \in \bbL'$ and $\pi_*(\bbL')=\bbL$.
Since $\widetilde C^2 = \tilde\nu(C)\ge0$,
\ref{dskjfwejf;akj}(b) implies that the general member of $\bbL'$
is a rational curve; so the general member of  $\pi_*(\bbL') = \bbL$ is a rational curve.
\end{proof}

\begin{proposition} \label {dpfioupq9034r2093iej}
Let $C \subset S$ be a curve and suppose that $\bbL$ is a rational linear
system on $S$ (cf.~\ref{88800112h}) satisfying $C \in \bbL$.
Then the following hold.
\begin{Enumerata}{1mm}

\item $C$ is a rational curve of nonnegative type.

\item $\widetilde C^{K_{\bbL}}  \in  \widetilde\bbL^{K_{\bbL}}$
and $\widetilde C^{K_{\bbL}}$ is nonsingular.

\item  $C$ goes through $\Keul_\bbL$ effectively.

\item The general member $D$ of $\bbL$ satisfies $\Keul^C = \Keul^D$.

\item $\bbL \subseteq \bbL_C$ 

\item $K^C \subseteq K_\bbL$ and
${\tilde \nu}(C) = ( \widetilde C^{K_\bbL} )^2 + | K_\bbL \setminus K^C | \ge | K_\bbL \setminus K^C |$.

\end{Enumerata}
\end{proposition}

\begin{proof}
There is a nonempty Zariski-open subset $U$ of $\bbL$ such that every
element of $U$ is an irreducible rational curve.
Pick a pencil $\Lambda \subseteq \bbL$ such that $C \in \Lambda$ and
$\Lambda \cap U \neq \emptyset$;
then $\Lambda$ is a rational pencil.
Let $\pi_{ K_{ \Lambda } } : S_{ K_{\Lambda} } \to S$
be the minimal resolution of the base points of $\Lambda$.
Then $\widetilde\Lambda^{K_{ \Lambda }}$ is a $\proj^1$-ruling
and $\widetilde C^{K_{ \Lambda }}$ is included in the support of an element
of $\widetilde\Lambda^{K_{ \Lambda }}$,
so Gizatullin's Theorem~\ref{jhdhsgghgaqaqoaiaaqo}
implies that $\widetilde C^{K_{ \Lambda }}$ is rational
(so $C$ is rational) and nonsingular (so $K^C \subseteq K_{\Lambda}$ 
by \eqref{dofj[q293i[lka;dfk;;}).
Let $F \in \widetilde\Lambda^{K_{\Lambda}}$
be the element such that $\widetilde C^{K_{\Lambda}} \subseteq \supp(F)$.
The fact that $1C \in \Lambda$ implies that
\begin{equation} \label{llqlqlqllqllqlqq}
F = 1 \widetilde C^{K_{\Lambda}}
+  \sum_{P \in I} a_P \widetilde E_P^{K_{\Lambda}}
\end{equation}
for some subset $I \subseteq K_\Lambda$ and where $a_P\ge1$ for 
all $P \in I$.

We claim that $I = \emptyset$.  Indeed, suppose the contrary.
Then $\supp(F)$ is reducible, so Gizatullin's Theorem implies that 
$\supp(F)$ has a $(-1)$-component $\Gamma$, and that if $\Gamma$
has multiplicity $1$ in $F$ then $\Gamma$ is not the only $(-1)$-component
of $\supp(F)$. This together with \eqref{llqlqlqllqllqlqq} imply that there
exists $P \in I$ such that $(\widetilde E_P^{K_{\Lambda}})^2 = -1$;
as $P\in K_{\Lambda}$ and $\widetilde E_P^{K_{\Lambda}}$ is vertical,
this contradicts \eqref{sdfjpqw88888}, and proves that $I=\emptyset$.
So:
\begin{equation} \label{oioqoiqoqoiqioqiqoi}
\widetilde C^{K_{\Lambda}}  \in  \widetilde\Lambda^{K_{\Lambda}} .
\end{equation}
It follows that
$(\widetilde C^{K_{\Lambda}})^2 = 0$  in $S_{K_{\Lambda}}$,
because $\widetilde\Lambda^{K_{\Lambda}}$ is a base-point-free pencil.
As $\widetilde C^{K_{\Lambda}}$ is also nonsingular,
$C$ is of nonnegative type and (a) is proved.

Since $\widetilde C^{K_{\Lambda}}  \in  \widetilde\Lambda^{K_{\Lambda}}$ and 
$\widetilde C^{K_{\Lambda}}$ is nonsingular, \ref{dfiupqwe;lkdj}
implies that $\Keul^C = \Keul_{\Lambda}^{(>1)}$.
As the general member $D$ of $\Lambda$ satisfies 
$\Keul^D = \Keul_{\Lambda}^{(>1)}$ by \ref{dfiupqwe;lkdj}, we get $\Keul^D = \Keul^C$.
So we have shown that, for any pencil
$\Lambda$ satisfying $C \in \Lambda \subseteq \bbL$
and $\Lambda \cap U \neq \emptyset$, the general member $D$ of $\Lambda$
satisfies $\Keul^D = \Keul^C$.
Consequently, $\setspec{ D \in \bbL }{ \Keul^D = \Keul^C }$ is a dense
subset of $\bbL$; together with the fact (\ref{dfiupqwe;lkdj}) that 
$\Keul^D = \Keul_\bbL^{(>1)}$ for general $D \in \bbL$, this implies
\begin{equation}  \label {difuq903erq98493`2}
\Keul^C = \Keul_\bbL^{(>1)} .
\end{equation}
Then assertions (b) and (d) follow from \eqref{difuq903erq98493`2} and
\ref{dfiupqwe;lkdj}, and assertion (c) follows from
$\widetilde C^{K_{\bbL}}  \in  \widetilde\bbL^{K_{\bbL}}$
and \ref{;sdkfj9238923-iklj}.
By  \eqref{difuq903erq98493`2} we have
$e_P(C) = e_P( \bbL )$ for all $P \in K^C$; this together with \ref{;sdkfj9238923-iklj} implies that
$\widetilde C^{K^{C}}  \in  \widetilde\bbL^{K^{C}}$, hence
$\widetilde\bbL^{K^{C}} \subseteq | \widetilde C^{K^{C}} |$.
It follows that $\pi_*\big( \widetilde\bbL^{K^{C}} \big) \subseteq \pi_*\big( | \widetilde C^{K^{C}} | \big)$,
where $\pi_* : \Div( S_{K^C} ) \to \Div(S)$ is the homomorphism induced by $\pi = \pi_{K^C} : S_{K^C} \to S$.
As $\pi_*\big( \widetilde\bbL^{K^{C}} \big) = \bbL$
and (by definition) $\pi_*\big( | \widetilde C^{K^{C}} | \big) = \bbL_C$, (e) is true.

We have $K^C \subseteq K_\bbL$ by \eqref{difuq903erq98493`2}, and (c) implies that
$e_P(C)=1$ for all $P \in K_\bbL \setminus K^C$.
Consequently, 
${\tilde \nu}(C) = ( \widetilde C^{K^C} )^2
= ( \widetilde C^{K_\bbL} )^2 + | K_\bbL \setminus K^C |$.
Pick any $D \in \widetilde\bbL^{K_{\bbL}} \setminus \{
\widetilde C^{K_{\bbL}}  \}$, then 
$( \widetilde C^{K_\bbL} )^2  = 
\widetilde C^{K_\bbL} \cdot D \ge0$ and (f) is proved.
\end{proof}

\begin{theorem}   \label {yyvyvyvyvydyddsdfy}
For a curve $C \subset S$, the following are equivalent:
\begin{enumerata}

\item $\bbL_C$ is a rational linear system;

\item there exists a rational linear system $\bbL$ on $S$
such that $C \in \bbL$;

\item $C$ is rational and $\dim\bbL_C \ge 1$;

\item $C$ is rational and ${\tilde \nu}(C) \ge 0$.

\end{enumerata}
Moreover, if conditions \mbox{\rm (a--d)}
are satisfied then the following hold:

\begin{enumerata}
\addtocounter{enumi}{4}

\item For a linear system $\bbL$ on $S$ satisfying
$C \in \bbL$ and $\dim \bbL \ge1$, 
$$
\text{$\bbL$ is rational } \iff  \bbL \subseteq \bbL_C .
$$

\item  $\dim\bbL_C = {\tilde \nu}(C)+1$, $\Keul_{\bbL_C} = \Keul^C$
and $\widetilde\bbL_C^{K^C} = | \widetilde C^{K^C} |$.

\end{enumerata}
\end{theorem}

\begin{proof}
Suppose that $\bbL$ is a rational linear system on $S$ such that $C \in \bbL$.
By \ref{dpfioupq9034r2093iej},
we obtain that $C$ is rational and that $\bbL \subseteq \bbL_C$ (so $\dim\bbL_C \ge 1$).
So (b) implies (c),
and this also proves implication ``$\Rightarrow$'' in statement~(e).

Equivalence (c)$\Leftrightarrow$(d) is \ref{bbbb3982edu}(a), 
implication (c and d)$\Rightarrow$(a) is the case $\bbL=\bbL_C$ of \ref{bbbb3982edu}(c),
and (a)$\Rightarrow$(b) is obvious.  So (a--d) are equivalent.

Now assume that (a--d) are satisfied.
Implication ``$\Leftarrow$'' in statement (e) is a consequence of \ref{bbbb3982edu}(c),
so there only remains to prove (f).
Equality $\dim\bbL_C = {\tilde \nu}(C)+1$ is \ref{bbbb3982edu}(b).
Observe that there can be at most one linear system $\bbG$ on $S_{K^C}$ satisfying
\begin{center}
the general member of $\bbG$ is irreducible and $\pi_*(\bbG) = \bbL_C$,
\end{center}
where $\pi = \pi_{K^C} : S_{K^C} \to S$;
as $\widetilde\bbL_C^{K^C}$ and  $| \widetilde C^{K^C} |$ are two such linear systems,
we get $\widetilde\bbL_C^{K^C} = | \widetilde C^{K^C} |$. This implies
that $\widetilde\bbL_C^{K^C}$ is base-point-free (because $| \widetilde C^{K^C} |$ is base-point-free by
\ref{dskjfwejf;akj}), so all base points of $\bbL_C$ are in $K^C$, i.e., $K_{\bbL_C} \subseteq K^C$. 
On the other hand, \ref{dpfioupq9034r2093iej}(b) together with \ref{dfiupqwe;lkdj} gives
$\Keul^C = \Keul_{\bbL_C}^{(>1)}$ (which was also noted in \eqref{difuq903erq98493`2}); this and 
$K_{\bbL_C} \subseteq K^C$ imply $\Keul_{\bbL_C} = \Keul^C$, which completes the proof.
\end{proof}


\end{document}